\documentclass[10pt,twoside]{amsart}
\usepackage[all]{xy}   
\usepackage {amssymb, latexsym, amsthm, amsmath, amsfonts, amscd, pgf, comment}

\long\def\symbolfootnote[#1]#2{\begingroup%
\def\thefootnote{\fnsymbol{footnote}}\footnote[#1]{#2}\endgroup}

\makeatletter
\def\imod#1{\allowbreak\mkern10mu({\operator@font mod}\,\,#1)}
\makeatother

\newtheorem{theorem}{Theorem}[section]
\newtheorem{proposition}[theorem]{Proposition}
\newtheorem{lemma}[theorem]{Lemma}

\newtheorem{remark}[theorem]{Remark}

\newcommand{\Z}{\mbox{$\mathbb Z$}}

\newcommand{\F}{\mbox{$\mathbb F$}}  
\newcommand{\N}{\mbox{$\mathbb N$}}     
\newcommand{\R}{\mbox{$\mathbb R$}}     

\def\remark{\paragraph{\bf Remark}}

\newcommand{\Or}{\operatorname{O}}

\newcommand{\charac}{\operatorname{char}}

\newcommand{\rad}{\operatorname{rad}}
\newcommand{\Arf}{\operatorname{Arf}}
\newcommand{\Quad}{\operatorname{Quad}}
\newcommand{\Hom}{\operatorname{Hom}}

\begin{document}
\setcounter{section}{0}
\title{Strongly real special $2$-groups}

\author{Dilpreet Kaur}
\address{Indian Institute of Science Education and Research, Knowledge City, Sector-81, Mohali 140306 INDIA.}
\thanks{The first named author gratefully acknowledges the Junior Research Fellowship from Council of Scientific \& Industrial Research, India
for the financial support during the period of this work.}
\email{dilpreetkaur@iisermohali.ac.in}

\author{Amit Kulsherstha}
\address{Indian Institute of Science Education and Research, Knowledge City, Sector-81, Mohali 140306 INDIA.}
\email{amitk@iisermohali.ac.in}

\date{}
\subjclass[2000]{20C15, 20C33}
\keywords{Strongly real, totally orthogonal}
\begin{abstract}
Strongly real groups and totally orthogonal groups form two important subclasses of real groups.
In this article we give a characterization of strongly real special $2$-groups.
This characterization is in terms of quadratic maps over fields of characteristic $2$.
We then provide examples of groups which are in one subclass and not the other. It is a conjecture of Tiep that such examples
are not possible for finite simple groups.
\end{abstract}

\maketitle

\section*{Introduction}\label{introduction}
Groups in which all elements belong to the conjugacy class of their inverses are called {\it real groups}.
It is well known that all entries in the character tables of finite real groups are real \cite{SerreRepnBook}.
However real groups may admit complex representations which are not realizable over
$\R$. Such representations are called {\it symplectic}. 
The complex representation
$$1 \mapsto \left(\begin{array}{cc} 1 & 0 \\ 0 & 1 \end{array}\right), 
~~a \mapsto \left(\begin{array}{cc} i & 0 \\ 0 & -i \end{array}\right),
~~b \mapsto \left(\begin{array}{cc} 0 & 1 \\ -1 & 0 \end{array}\right),
~~c \mapsto \left(\begin{array}{cc} 0 & 1 \\ i & 0 \end{array}\right)$$
of the Quaternion group $Q_2 = \langle -1, a, b, c : (-1)^2 = 1, a^2 = b^2 = c^2 = abc = -1 \rangle$
is symplectic.
Real groups which do not admit any symplectic representation are called {\it totally orthogonal}.
In 1985, Gow proved that groups $\Or_n(q)$ of special isometries of quadratic forms
are totally orthogonal \cite{GowOrthogonal}. It was already known by then that
every element of $\Or_n(q)$ is strongly real \cite{WonenburgerOrthogonal}. {\it Strongly real} elements in a group are
those which can be expressed as a product of two self-inverses. 
A group is called {\it strongly real}
if all its elements are strongly real. 
It is straightforward to see that a group $G$ is strongly real if for every element $x$ of $G$, there exists an element $y$
in $G$ such that $y^2 = 1$ and $x^{-1}=y^{-1}xy$. \\

Like $\Or_n(q)$ there are plenty of groups which are both
totally orthogonal and strongly real. 
According to a conjecture of Tiep finite simple groups are totally orthogonal if and only if they are strongly real.
However there is no general class of groups known which exhibits one of these properties and not the other.
In this article we exhibit an infinite class of groups, namely the class of special $2$-groups,
for which none of the notions of strongly reality and total orthogonality imply the other. 
This generalizes an example of a strongly real group admitting a symplectic representation \cite{AmitAnupam}. \\

The plan of the article is as follows: In \S 1 we recall basics on quadratic forms over $\F_2$, where $\F_2$ denotes the field with two elements. Then in \S 2
we record special $2$-groups, quadratic maps, second cohomology groups and their interconnection. The understanding of
this interconnection gives us a characterization of strongly real special groups in terms of their associated quadratic maps. 
This characterization is used in next section to show that all extraspecial $2$-groups, except the quaternion group $Q_2$, 
are strongly real. This requires a classification of extraspecial $2$-groups as central products of $D_4$ and $Q_2$, where $D_4$ denotes
the dihedral group of order $8$ \cite{GorensteinBook}.
In the same section we convert this classification in the language of quadratic forms over $\F_2$.
In last section we give examples of groups which are strongly real but not totally orthogonal, and vice-versa. 
A characterization of strongly real groups (Th. \ref{strongly-real-criterion}) 
and a characterization of totally orthogonal groups \cite[Th. 3.5]{ObedPaper}
play key roles in the construction of these examples.

\section{Quadratic forms in characteristic $2$}
Let $\F$ be a field of characteristic two and $V$ be an $n$ dimensional vector space over $\F$. 
A map $b: V\times V\rightarrow \F$ is called {\it $\F$-bilinear} if it satisfies the following properties:
\begin{enumerate}
  \item $b(\alpha v_1+\beta v_2, w)=\alpha b(v_1,w)+\beta b(v_2,w)$ for all $v_1, v_2, w \in V$
and $\alpha, \beta \in \F$.
  \item $b(v, \alpha w_1+\beta w_2)=\alpha b(v,w_1)+\beta b(v,w_2)$ for all $v, w_1, w_2 \in V$
and $\alpha, \beta \in \F$
\end{enumerate}
A map $q:V\rightarrow \F$ is called {\it quadratic form} if 
\begin{enumerate}
  \item $q(\alpha v)=\alpha^2 q(v)$ for all $v \in V$
  \item The map $b_q: V\times V\rightarrow \F$ given by $b_q(v,w):=q(v+w)-q(v)-q(w)$ is $\F$-bilinear.
\end{enumerate}
For a quadratic form $q$ the map $b_q$ is called the {\it polar map} of $q$. The pair $(V,q)$ is called
the {\it quadratic space} over $\F$. A bilinear map $b_q$ is called {\it alternative} if $b_q(v,v)=0$
for all $v \in V$.\\

If $B = \{e_1,e_2,\cdots,e_n\}$ is a basis of $V$ then any $n \times n$ matrix $Q$ satisfying 
$q(x)=x^t Qx$, where $x\in V$ is indeterminate column vector and $x^t$ denotes the transpose of $x$,
is called a {\it matrix of $q$ with respect to basis $B$}. 
Every matrix of $q$ with respect to same basis is of form $Q + A$, where $A$ is an alternating matrix and 
$Q$ is the unique upper triangular matrix of $q$ with respect to basis $B$. If we change the basis 
and $T$ is transition matrix for the change of basis then 
upper triangular matrix of $q$ with respect to new basis is $T^t QT$. \\

Two $n$-dimensional quadratic spaces $(V_1, q_1)$ and $(V_2, q_2)$ over $\F$ are said to be 
{\it isometric} if there exists an $\F$-linear isomorphism $T:V_1\rightarrow V_2$ such that 
$q_1(v)=q_2(T(v))$ for all $v\in V$. Isometry between two quadratic spaces is denoted by 
$(V_1, q_1)\cong(V_2, q_2).$ 
A quadratic space $(V, q)$ is called the {\it orthogonal sum} of $(V_1, q_1)$ and $(V_2, q_2)$ if 
$V = V_1\oplus V_2$ and $q(v)=q_1(v_1)+q_2(v_2)$, where $v = (v_1, v_2) \in V$ is an arbitrary element of
$V$. In this case we write $q=q_1 \bot q_2$. Conversely, let $(V,q)$ be a quadratic space and $V_i,~1 \leq i\leq m$ be subspaces of $V$ such that $V = V_1\oplus \cdots \oplus V_m$ and $b_q(v_i,v_j) = 0$ for $v_i\in V_i, v_j\in V_j, i\neq j$. Then $q = q_1 \bot \cdots \bot q_m$ where $q_i$ denotes the restriction of $q$ to $V_i$. \\
 
The subspace $\rad(V) := \{w\in V : b_q(v,w)=0 ~\forall v\in V\}$ is called the {\it radical} of $(V,q)$. The quadratic space $(V,q)$ is called {\it regular} if $\rad(V)=0$. \\

The following theorem is analogous to the diagonalization of quadratic spaces over the field of characteristic different from $2$.\\
\begin{theorem}[\cite{Pfister}]\label{quad-form-decompo-U-and-rad(V)}
Every quadratic space $(V,q)$ has an orthogonal decomposition $V = U \oplus \rad(V)$ such that $(U,q|_U)$ is regular and an orthogonal 
sum of $2$-dimensional regular quadratic spaces, where $(\rad(V), q|_{\rad(V)})$ is an orthogonal sum of $1$-dimensional quadratic spaces.
\end{theorem}

More explicitly, the orthogonal decomposition $V = U \oplus \rad(V)$ is such that there exists a basis $\{e_i, f_i, g_j, 1\leq i\leq r, 1\leq j\leq s\}$ of $V$ where $2r+s=\dim(V)$ and elements 
$a_i, b_i, c_j \in \F,1\leq i\leq r, 1\leq j\leq s$ such that for all $\displaystyle v=\sum_{i = 1}^r (x_i e_i + y_i f_i)+\sum_{i = 1}^s z_j g_j,$ we have
$$q(v)=\sum_{i = 1}^r (a_i x_{i}^{2}+x_i y_j+ b_i y_{i}^{2})+\sum_{i = 1}^s c_j z_{j}^{2}.$$

In this case we say that $[a_1,b_1]\bot \cdots \bot [a_r, b_r]\bot \langle c_1,\cdots,c_s \rangle$ is the
{\it normalized form} of $q$ and $[a_1,b_1]\bot \cdots \bot [a_r, b_r]$ is the {\it regular part} of $q$. 
A quadratic form $q$ is said to be {\it non-singular} if $s=0$. It is called {\it singular} if $s\neq0$. In addition, if $r=0$ then $q$ 
is called {\it totally singular}.
If $s>0$ then regular part of quadratic form is generally not determined uniquely up to isometry, whereas the part $\langle c_1,\dots,c_2\rangle$ is 
always determined uniquely up to isometry. For example $[1,1]\oplus \langle 1 \rangle \cong [0,0]\oplus \langle 1 \rangle$ but $[1,1]\cong [0,0]$ holds if and only if the quadratic equation $x^2+x+1=0$ has a solution in $\F$. \\

It immediately follows from above Proposition that every regular quadratic form over a field of characteristic two is even dimensional.\\

 A quadratic from is said to be {\it isotropic} if there exists $0\neq v\in V$ such that $q(v)=0$, otherwise it is called {\it anisotropic}.
Quadratic form $[0,0]$ is the only $2$-dimensional isotropic quadratic form up to isometry. It is called {\it hyperbolic plane} and is denoted by $H$.
A quadratic space is said to be a {\it hyperbolic space} if it is orthogonal sum of hyperbolic planes. \\

  The following result is an analogue in characteristic 2 to the usual Witt decomposition in characteristic not equal to  2.
\begin{proposition}[\cite{HL}]\label{witt-decomposition-in-char-2}
Let $q$ be a quadratic form over $\F$. Then $q = i\times H ~\bot~ q_r~\bot~ q_s ~\bot~ j\times \langle 0 \rangle$, where 
$q_r$ is non-singular, $q_s$ is totally singular and $q_r ~\bot~ q_s$ is anisotropic. The form $q_r~\bot~ q_s$ is uniquely determined up to isometry.
\end{proposition}
The anisotropic part $q_r ~\bot~ q_s$ of $q$ will be denoted by $q_{an}$ in short. Two quadratic forms $q_1$ and $q_2$ are called
{\it Witt equivalent} (denoted $q_1 \sim q_2$)  if ${q_1}_{an}\cong {q_2}_{an}$. If $q_1$ and $q_2$ are non-singular 
then $q_1 \sim q_2$ if $q_1 ~\bot~ -q_2$ is hyperbolic. The set $W_q(\F)$ of Witt equivalence classes of regular quadratic forms
over $\F$ forms an abelian group under the operation of orthogonal sum of quadratic spaces. It is called the {\it Witt group of 
quadratic forms}. \\
  
As all regular quadratic forms are of even dimension, the {\it dimension invariant}
$e_0 : W(\F) \to \frac{\Z}{2\Z}$ given by $q \mapsto \dim(q)\mod 2$ is trivial. In $\charac(\F) = 2$ case an invariant at next level is 
Arf invariant which is a substitute of Discriminant in $\charac(\F) \neq 2$ case. It was defined by Arf in his classical paper \cite{Arf}.
  Let $q: V \rightarrow \F$ be a regular $2n$-dimensional quadratic form. As $b_q$ is an alternating form, the space $(V, b_q)$ has a symplectic basis
$\{e_i, f_i, 1\leq i\geq n\}$. Let
$\wp(\F)=\{x^2+x : ~x\in \F\}$. The set $\wp(\F)\}$ is a subgroup of $(\F,+)$.
In the quotient $\F/\wp(\F)$, the class of element $\sum_{i = 1}^n q(e_i)q(f_i)$ is called the {\it Arf invariant of $q$} and is denoted by $\Arf(q)$. 
It is independent the choice of symplectic basis (see \cite{Scharlau}). Moreover, it defines a homomorphism $\Arf : W_q(\F) \to \F/\wp(\F)$ 
on the Witt group of $\F$. More explicitly if $q=[a_1,b_1]~\bot~ \cdots ~\bot~ [a_n,b_n]$ then $\Arf(q) = a_1b_1 + \cdots + a_nb_n \in \F/\wp(\F)$. \\

\section{Special $2$-groups and quadratic forms}   
Let $G$ be a finite group. Let $\Omega (G)=\langle x\in G~:~x^2=1\rangle$ and $G^{\prime} = [G, G]$, the derived subgroup of $G$. 
Let $Z(G)$ denote the center of $G$ and $\Phi(G)$ denote the Frattini subgroup of $G$. Recall that for a non-trivial group 
the intersection of its index $2$ subgroups is called its {\it Frattini subgroup}. We record that 
$\Phi(G) =\langle x^2 : x\in G \rangle$. \\

A $2$-group $G$ is called {\it special $2$-group} if $G$ is non-commutative and $$G^{\prime} = \Phi (G) = Z(G) = \Omega (Z(G)).$$
Moreover, if $G$ is a special $2$-group such that $|Z(G)|=2$ then $G$ is called an {\it extraspecial $2$-group}. \\
\begin{remark}\label{special-2-group-order-2-or-4}
In a special $2$-group, the order of non-identity elements is either $2$ or $4$.  \\
\end{remark}

From now onwards $\F_2$ will denote the field of two elements and $V$ and $W$ will be vector spaces over $\F_2$.
A map $c: V \times V \rightarrow W$ is called a {\it normal $2$-cocycle} on $V$ with coefficients in $W$ if 
for all $v,v_1,v_2,v_3\in V$ it satisfies the following conditions:
\begin{enumerate}
\item[$i.$] $c(v_2,v_3)-c(v_1+v_2,v_3)+c(v_1,v_2+v_3)-c(v_1,v_2)=0$.
\item[$ii.$] $c(v,0)=(0,v)=0$. \\
\end{enumerate} 

We denote the set of normal $2$-cocycles on $V$ with coefficients in $W$ by $Z^2(V,W)$ and consider it as an abelian group
under pointwise addition.
Let $\lambda:V\rightarrow W$ be a linear map such that $\lambda(0)=0$. Then the map
$c_{\lambda} : V \times V \to W$ defined by $c_{\lambda}(v_1,v_2)=\lambda(v_2)-\lambda(v_1+v_2)+\lambda(v_1)$
is a normal $2$-cocycle. Such $2$-cocycles are called {\it normal $2$-coboundaries} and their collection is denoted by $B^2(V,W)$.
The set $B^2(V,W)$ forms a subgroup of $Z^2(V,W)$. The quotient $H^2(V,W)=\frac {Z^2(V,W)}{B^2(V,W)}$ is called
the {\it second cohomology group of $V$ with coefficients in $W$}. \\

We consider $V$ and $W$ as groups under the operation of vector space addition. A short exact sequence of groups 
$1\rightarrow V\rightarrow G\rightarrow W\rightarrow 1$ is called a {\it central extension of $W$ by $V$} if
$V\subseteq Z(G)$. The set of isomorphism classes of central extension of $W$ by $V$ is in one to one correspondence with $H^2(V,W)$ 
\cite[\S 6.6]{Weibelbook}. We record that the central extension of $W$ by $V$ corresponding 
to a cocycle class $[c]\in H^2(V,W)$ is isomorphic to the group $V\dot{\times} W$, where the underlying set of
the group $V\dot{\times} W$ is just the cartesian product $V \times W$ and its group operation is defined by
\begin{center}
 $(v,w)(v^{\prime},w^{\prime})=(v+v^{\prime}, c(v,v^{\prime})+w+w^{\prime})$
\end{center}
for all $v, v^{\prime} \in V$ and $w, w^{\prime} \in W$.
The identity element this group is $(0,0)$ and the inverse of $(v,w)$ is $(v, c(v,v)+w)$. \\

A map $q : V \to W$ is called a {\it quadratic map} if $q(\alpha v) = \alpha^2 q(v)$ for all $v \in V$, $\alpha \in \F$
and the map $b_q : V \times V \to W$ defined by $b_q (v, w) = q(v+w)- q(v) - q(w)$ is bilinear.
We denote by $\langle b_q(V \times V) \rangle$ the subgroup of $W$ generated by the image of $b_q$.
Let $\Quad(V,W)$ denote the set of quadratic maps from $V$ to $W$. We consider it as a group under the group operation
of pointwise addition of maps. The following proposition gives correspondence between $H^2(V,W)$ and $\Quad(V,W)$.
\begin{proposition}[\cite{ObedPaper}, Prop. 1.2] \label{cohomology-and-quadratic-map-correspondance}
 The map $\varphi: Z^2(V,W) \rightarrow \Quad(V,W)$ which maps $c\in Z^2(A,B)$ to the quadratic map $q_c$ defined by $q_c(x)=c(x,x)$
induces a homomorphism between $H^2(V,W)$ and $\Quad(V,W)$. If the dimension of $V$ is finite then this homomorphism is
an isomorphism. 
\end{proposition}

If dimension of $V$ is finite then above proposition and the correspondence of elements of $H^2(V,W)$ with central extension of $V$
by $W$ gives a useful correspondence between $\Quad(V,W)$ and central extension of $V$ by $W$ \cite{ObedPaper}.

\begin{theorem}[\cite{Obedthesis}, Th. 3.4.11] \label{quad-map-of-special-2-group}
Let $G$ be a special $2$-group and $q:\frac{G}{Z(G)}\rightarrow Z(G)$ be the map given by $q(xZ(G))=x^2$ for all $x\in G$. Then
$q$ is a quadratic map and $b_q(xZ(G),yZ(G))=xyx^{-1}y^{-1}$ for all $x,y\in G$. 
\end{theorem}

Note that the quadratic map $q$ in the above theorem is regular and the image of $b_q$ generates $Z(G)$.
This quadratic map $q$ is called the {\it quadratic map associated to the special $2$-group $G$}. 

\begin{theorem}[\cite{ObedPaper}, Th. 1.4] \label{special-2-group-of-a-quad-map}
Let $q:V\rightarrow W$ be a regular quadratic map. Suppose that $ W=\langle b_q(V\times V)\rangle$.
Then there exists a special $2$-group $G$ associated with quadratic map $q$ such that $W=Z(G)$ and $V=\frac{G}{Z(G)}$. Such a group is unique up to isomorphism.
\end{theorem}

%

The special $2$-group $G$ in the above theorem is called {\it the group associated to regular quadratic form $q$}.


We recall the definition of central product of a two groups. Let $G_1$ and $G_2$ two groups, $Z(G_1)$ and $Z(G_2)$ be
their centers and $\theta : Z(G_1) \to Z(G_2)$ be a group isomorphism. 
Let $N$ denote the normal subgroup $\{(x, y) \in Z(G_1) \times Z(G_2) : \theta(x)y = 1\}$ of $Z(G_1) \times Z(G_2)$.
The {\it central product} of $G_1$ and $G_2$ is the quotient of direct product $G_1 \times G_2$ by $N$. 
The following lemma relates orthogonal sum of quadratic maps to central products.

\begin{lemma}\label{ortho-sum-corresponds-to-central-product}
Let $V_1, V_2$ and $W$ be finite dimensional vector spaces over the field $\F_2$.
Let $q_{1}:V_{1}\rightarrow W$ and $q_{2} : V_{2}\rightarrow W$ be regular quadratic
maps associated to special $2$-groups $G_{1}$ and $G_{2}$, respectively. Then 
$q_{1}\perp q_{2} : V_1 \oplus V_2 \to W$ defined by 
$(q_1 \perp q_2)(v_1, v_2) = q_1(v_1) + q_2(v_2)$ is a regular quadratic map and the group associated
to $q_{1}\perp q_{2}$ is $G_{1}\circ G_{2}$, where $\circ$ denotes the central product.
\end{lemma}
\begin{proof}
Since quadratic maps $q_{1}:V_{1}\rightarrow W$ and $q_{2} : V_{2}\rightarrow W$ are associated to
special $2$-groups $G_1$ and $G_2$, respectively, we have
$V_1=\frac{G_1}{Z(G_1)},~V_2=\frac{G_2}{Z(G_2)}$ and $W=Z(G_1)=Z(G_2)$. 
Let $c_1\in Z^2(V_1,W)$ and $c_2\in Z^2(V_2,W)$ be normal $2$-cocycles associated to quadratic maps
$q_1$ and $q_2$, respectively.\\

Let $q =q_1 \perp q_2$. Then $b_q((v_1,v_2),(v_1^{\prime},v_2^{\prime}))=
b_{q_{1}}(v_1,v_1^{\prime})+b_{q_{2}}(v_2,v_2^{\prime})$ where $
(v_1, v_2)$ and $(v_1^{\prime},v_2^{\prime})\in V_1\oplus V_2$.
Let $c := c_1 \perp c_2 : (V_1\oplus V_2) \times (V_1 \oplus V_2) \to W$ be the map defined by
$$c((v_1,v_2),(v_1^{\prime},v_2^{\prime}))= c_1(v_1,v_1^{\prime})+c_2(v_2,v_2^{\prime}).$$
It is straightforward to check that $c$ is a normal $2$-cocycle on $V_1 \oplus V_2$ with 
coefficients in $W$ and the association $([c_1], [c_2]) \mapsto [c] \in H^2(V_1 \oplus V_2, W)$ is well-defined.
The normal $2$-cocycle $c$ corresponds to the quadratic form $q$. Special $2$-group associated with $q$
is $G=(V_1\oplus V_2)\dot{\times}W$, with the group operation
\begin{center}
$((v_1,v_2),w)((v_1^{\prime}, v_2^{\prime}),w^{\prime})=((v_1+v_1^{\prime}, v_2+v_2^{\prime}), 
c((v_1, v_2), (v_1^{\prime},v_2^{\prime}))+w+w^{\prime})$.
\end{center}
We need to show that $G = G_1 \circ G_2$. By definition $G_1 \circ G_2$ is the quotient of $G_1 \times G_2$ by
$\ker(f)$ where $f:W\times W\rightarrow W$ is the homomorphism defined by $f(w,w^{\prime})=w+w^{\prime}$; $w,w^{\prime}\in W$.
Define $\phi :G_1\times G_2\rightarrow G$ by
\begin{center}
 $\phi((v_1,w),(v_2,w^{\prime}))=((v_1,v_2),w+w^{\prime})$
\end{center}
where $(v_1,w)\in G_1,(v_2,w^{\prime})\in G_2$. We notice that $\phi$ is a group homomorphism as
\begin{align*}
&\quad \quad \phi(((v_1,w_1),(v_2,w_1^{\prime}))+((v_1^{\prime},w_2),(v_2^{\prime},w_2^{\prime}))) \\
&=\phi(((v_1,w_1)+(v_1^{\prime},w_2)),((v_2,w_1^{\prime})+(v_2^{\prime},w_2^{\prime})))\\
&=\phi((v_1+v_1^{\prime},c_1(v_1,v_1^{\prime})+w_1+w_2),(v_2+v_2^{\prime},c_2(v_2,v_2^{\prime})+w_1^{\prime}+w_2^{\prime}))\\
&=((v_1+v_1^{\prime}, v_2+v_2^{\prime}), c_1(v_1,v_1^{\prime})+c_2(v_2,v_2^{\prime})+w_1+w_1^{\prime}+w_2+w_2^{\prime})\\
&=((v_1,v_2) + (v_1^{\prime},v_2^{\prime}), c((v_1,v_2),(v_1^{\prime},v_2^{\prime}))+w_1+w_1^{\prime}+w_2+w_2^{\prime})\\
&=((v_1,v_2),w_1+w_1^{\prime})((v_1^{\prime},v_2^{\prime}),w_2+w_2^{\prime})\\
&=\phi((v_1,w_1),(v_2,w_1^{\prime}))\phi((v_1^{\prime},w_2),(v_2^{\prime},w_2^{\prime})).
\end{align*}
where $(v_1,w_1),(v_1^{\prime},w_2)\in G_1$ and $(v_2,w_1^{\prime}),(v_2^{\prime},w_2^{\prime})\in G_2$. 
The homomorphism $\phi$ is surjective because for an arbitrary $((v_1,v_2),w)\in G$
we have $\phi((v_1,0),(v_2,w))=((v_1,v_2),w)$. 
Now identifying $W\times W$ with $(0\dot{\times}W)\times (0\dot{\times}W)\subset G_1\times G_2$, it follows that $\ker{\phi}$
gets identified with $\ker{f}$ and we finally have
$G \simeq \frac{G_1\times G_2}{\ker(\phi)} = \frac{G_1\times G_2}{\ker(f)} = G_1 \circ G_2$.
\hfill $\square$ \\
\end{proof}
We shall use this lemma while discussing classification of extraspecial $2$-groups. The following theorem gives a characterization of
strongly real special $2$-groups.
\begin{theorem}\label{strongly-real-criterion}
Let $q:  V \to W$ be a regular quadratic map with $\langle b_q(V\times V)\rangle = W$ and $G$ be the special $2$-group associated with $q$ such that $V = \frac{G}{Z(G)}$ and $W = Z(G)$ (cf. Th. \ref{special-2-group-of-a-quad-map}).
Then $G$ is strongly real if and only if for every nonzero $v\in V$ there exists $a\in V$ with $v \neq a$ and 
$q(a) = q(a - v) = 0$. 
\end{theorem}
\begin{proof}
We first suppose $G$ to be strongly real.
Let $0  \neq \overline{x} = v \in V$. Since $G$ is strongly real there exists $y \in G$ such that 
$o(y)=2$ and $yx^{-1}=xy$. We take $a = \overline{y}$. Then
$(yx^{-1})^2 = yx^{-1}xy = y^{2} = e$, which translates to
$q(a - v) = q(a) = 0$. \\


For the converse part, recall that $G = V\dot{\times} W$ where the group operation is defined by
\begin{center}
$(v,w)(v^{\prime}  ,w^{\prime} )=(v+v^{\prime}  ,c(v,v^{\prime})+w+w^{\prime})$\\
$(v,w)^{-1}=(v,c(v,v)+w)$.
\end{center}
where $[c] \in H^{2}(V,W)$ and $q(x)=c(x,x)$. Let $x = (v, w) \in V \dot{\times} W = G$. By hypothesis
there exists $a \in V$ such that $q(a) = q(a-v) = 0$. We take $y = (a - v, 0) \in G$.
Since $$(a - v, 0) + (a - v, 0) = (2(a-v), c(a-v,a-v)) = (2(a-v), q(a-v)) = (0, 0)$$
it follows that $y^2 = 1$. Moreover
\begin{align*}
(a-v, 0)(v,w)(a-v, 0) &=(2a - v, c(a-v,v)+c(a,a-v)+w)\\
&=(v, c(v,v) + c(a,a)+w) \\
&=(v, q(a) + c(v,v) + w) \\
&=(v, c(v,v) + w) = (v,w)^{-1}.
\end{align*}
Therefore $yxy = x^{-1}$. Further since $y^2 = 1$, we conclude that $yxy^{-1} = x^{-1}$
and $G$ is strongly real.
\hfill $\square$\\
\end{proof} 

In view of the above theorem, we remark that for a strongly real special $2$-group
the associated quadratic map is always isotropic. However the converse is not true.
For example, consider the special $2$-group $G$ associated to the quadratic map
$q(x,y,z)=(x^2+xy+y^2, xz)$. The quadratic map is isotropic because $q(0,0,1)=(0,0)$.
But the group $G$ is not strongly real by above theorem because for $v=(1,1,1)$, 
there does not exist any $a$ such that $q(a)=q(a-v) = 0$.

\section{Classification of extraspecial $2$-groups}
The aim of this section is to classify extraspecial $2$-groups in terms of
quadratic forms over $\F_2$ and to show that all extraspecial $2$-groups except $Q_2$ are
strongly real. We start with two quick lemmas.
\begin{lemma}\label{for-quad-form-extraspecial-group}
Let $V$ be a vector space over $\F_2$ and $q : V\rightarrow \F_2$ be a regular
quadratic form. Then the group associated to $q$ is extraspecial $2$-group,
and conversely.
\end{lemma}

\begin{proof}
Recall that a special $2$-group $G$ is called extraspecial $2$-group if $|Z(G)|=2$. Proof follows from Th. \ref{special-2-group-of-a-quad-map} and
Th. \ref{quad-map-of-special-2-group}.
\hfill $\square$
\end{proof}


\begin{lemma}\label{order-of-extraspecial-group-is-odd-power-of-2}
The order of an extraspecial group is $2^{2n+1}$ for some $n\in \N$. 
\end{lemma}

\begin{proof}
Recall that a regular quadratic form over a field of characteristic two is even dimensional (see \S 2, Th. \ref{quad-form-decompo-U-and-rad(V)}). 
If $G$ is an extraspecial $2$-group and $q: V\rightarrow \F_2$ is the regular quadratic form associated to it as in Th. \ref{quad-map-of-special-2-group},
then $\dim_{\F_2}V=2n$ for some $n\in \N$. From \S 3, $G$ is in bijection with $V{\times}\F_2$. Hence $|G|=2^{2n}\times 2=2^{2n+1}$.   
\hfill $\square$\\
 \end{proof}

In the rest of this article, $D_4$ will denote the dihedral group of order $8$ and $Q_2$ will denote the quaternion group of order $8$.
Presentations of these groups are
\begin{center}
$D_4=\langle a,b~ : ~a^4=b^2=1,bab^{-1}=a^{-1}\rangle$\\
$Q_2=\langle c,d~ : ~c^4=1,d^2=c^2,dcd^{-1}=c^{-1}\rangle$
\end{center}

\begin{proposition}\label{regular-forms-correspond-to Dihedral-and-Quaternion}
Let $V$ be the two dimensional vector space over $\F_2$ and $q:V\rightarrow \F_{2}$ be a regular quadratic form.
Then the group associated to $q$ is either $Q_{2}$ and $D_{4}$.
\end{proposition}

\begin{proof}
Up to isometry there are only two regular $2$-dimensional quadratic forms over $\F_{2}$. These are 
$q_1 = [0,0]$ and $q_2 = [1,1]$ (see \S 2).
Recall that $q_{1}(x,y)=xy$ and $q_{2}(x,y)=x^{2}+xy+y^{2}$.
We show that the extraspecial $2$-group corresponding to $[0,0]$ is $D_4$. 
From Th. \ref{special-2-group-of-a-quad-map} the group associated to $q_{1}:V \rightarrow \F_2$ is $V \dot{\times} \F_2$, where the multiplication is defined by
\begin{center}
 $(v,\alpha) (v^{\prime},\alpha^{\prime})=(v+v^{\prime}, c_1(v, v^{\prime})+\alpha+\alpha^{\prime})$
\end{center}
where $v,v^{\prime} \in V$, $\alpha,\alpha^{\prime}\in \F_2$ and $c_1\in Z^2(V,\F_2)$ is the normal $2$-cocycle
such that $q_1(v)=c_1(v,v)$ for all $v\in V$. \\

Define $\psi:D_4\rightarrow V \dot{\times} \F_2$ by $\psi(a)=((1,1),1)$ and $\psi(b)=((1,0),1)$, where $a$ and
$b$ are generating elements of $D_4$ as in the presentation of $D_4$. We check that $\psi$ is an
isomorphism of groups. Note that both $D_4$ and $V_{1}\dot{\times}\F_2$ are groups of order $8$.
It is easy to see that the orders of $((1,1),1)$ and $((1,0),1)$ are $4$ and $2$, respectively.
Moreover,
\begin{align*}
 ((1,0),1)((1,1),1)^{-1}&=((1,0),1)((1,1),c_1((1,1),(1,1))+1)\\
&= ((1,0),1)((1,1),q_1(1,1))+1)\\
&= ((1,0),1)((1,1),0)\\
&= ((1,0),c_1((1,0),(1,1))+1)\\
&= ((0,1),1)\\
&= ((1,1),1)((1,0),1)
\end{align*}
Hence $\psi$ is an isomorphism of groups.
On similar lines it can be shown that group associated with quadratic form $q_{2} = [1,1]$ is $Q_{2}$. The isomorphism is given by $\psi^{\prime}: Q_2\rightarrow V_{2}\dot{\times}\F_2$, where $\psi^{\prime}(c)=((1,1),1)$, $\psi^{\prime}(d)=((1,0),1)$ and $c$, $d$ denote generators of group $Q_2$ as in 
the given presentation of $Q_2$.
\hfill $\square$
\end{proof}

\begin{proposition}\label{[0,0]+[0,0]=[1,1]+[1,1]}
Let $q _{1}=[0,0]\bot [0,0]$ and $q_{2}=[1,1]\bot [1,1]$ be two quadratic forms over
$\F_2$. Then $q_1$ is isometric to $q_2$.
\end{proposition}

\begin{proof}
 We have $q _{1}(w,x,y,z)=wx+yz$ and $q _{2}(w,x,y,z)=w^{2}+wx+x^{2}+y^{2}+yz+z^{2}$.
The following change of variables in $q_{1}$ converts it to the form $q_2$
\begin{align*}
w &\mapsto x+y+z \\
x &\mapsto w+y+z \\
y &\mapsto w+x+z \\
z &\mapsto w+x+y
\end{align*}
\hfill $\square$
\end{proof}

\begin{proposition} \label{D-4-o-D-4-is-isometric-to-Q-2-o-Q-2}
The group $D_4\circ D_4$ is isomorphic to $Q_2\circ Q_2$ where $\circ$ denotes the central product of groups. 
\end{proposition}

\begin{proof}
From Lemma \ref{ortho-sum-corresponds-to-central-product} and Prop. \ref{regular-forms-correspond-to Dihedral-and-Quaternion}, it follows 
that the quadratic form associated to $D_4\circ D_4$ is $[0,0]\perp [0,0]$ and that the quadratic form
associated to $Q_2\circ Q_2$ is $[1,1]\perp [1,1]$. By Prop. \ref{[0,0]+[0,0]=[1,1]+[1,1]} we have that $q_1$ and $q_2$ are isometric
and now from Th. \ref{special-2-group-of-a-quad-map} the result follows.
\hfill $\square$
\end{proof}

\begin{theorem}\label{classification-of-exrtraspecial-2-groups}
For every $n\in \N$, there are exactly two extraspecial $2$-groups of order $2^{2n+1}$, namely
$D_{4}\circ D_{4}\circ \cdots \circ D_{4}$ ($n$ copies of $D_4$) and 
 $Q_{8}\circ D_{4}\circ \cdots\circ D_{4}$ ($n-1$ copies of $D_4$).
\end{theorem}

\begin{proof}
Let $G$ be an extraspecial $2$-group and $q : V \to \F_2$ be the associated regular quadratic form. 
Since $q$ is regular, $\dim_{\F_2}(V)$ is even, say $\dim_{\F_2}(V) = 2n$. Writing $q$ as orthogonal
sum of two dimensional regular spaces and using the isometry
$[0,0]\perp [0,0] \simeq [1,1] \perp [1,1]$ (see Prop. \ref{[0,0]+[0,0]=[1,1]+[1,1]}) we conclude that either
$q \simeq [0,0] \perp [0, 0] \perp \cdots \perp [0, 0]$ or $q \simeq [1,1] \perp\ [0,0] \perp \cdots [0, 0]$.
Thus, in view of Prop. \ref{regular-forms-correspond-to Dihedral-and-Quaternion}, the group corresponding to $q$ is either 
$D_{4}\circ D_{4}\circ \cdots \circ D_{4}$ ($n$ copies of $D_4$) or $Q_{2}\circ D_{4}\circ \cdots\circ D_{4}$ ($n-1$ copies of $D_4$).
\hfill $\square$ \\
 \end{proof}

This completes the classification of extraspecial $2$-groups. Their classification is also given in \cite{GorensteinBook} using group
theoretic methods. We learn from \cite{Wilson} that the classification of extraspecial $2$-group using quadratic forms is known.
However we do not know of any reference where it is mentioned in as much detail. \\

We shall denote the extraspecial 2-group $D_{4}\circ D_{4}\circ \cdots \circ D_{4}$ ($n$ copies of $D_4$) by 
$D_4^{(n)}$ and the extraspecial $2$-group $Q_{2}\circ D_{4}\circ \cdots\circ D_{4}$ ($n-1$ copies of $D_4$)
by  $Q_2D_4^{(n-1)}$. We now study strong reality of these groups.

\begin{lemma}\label{central-product-of-strongly-real-groups}
Central product of two strongly real groups is a strongly real group.
\end{lemma}
\begin{proof}
Central products are quotients of direct products. The direct product of two strongly real groups is a strongly real group.
Further, a quotient of strongly real group is a strongly real group. Hence the result follows.
\hfill $\square$ 
\end{proof}

\begin{proposition}\label{extraspecial-are-strongly-real}
All extraspecial 2-groups except $Q_{2}$ are strongly real.
\end{proposition}
\begin{proof}
It is easy to verify that the Dihedral group $D_{4}$ is strongly real. By Lemma. \ref{central-product-of-strongly-real-groups}, the groups
$D_4^{(n)}$, $n \in \N$ are strongly real. To show that groups $Q_2D_4^{(n-1)}$ are strongly real, by repeated use of 
Lemma \ref{central-product-of-strongly-real-groups} it is enough to show that $Q_{2}\circ D_{4}$ is strongly real.
We shall obtain this from Th. \ref{strongly-real-criterion}.
The quadratic form associated to $Q_{2}\circ D_{4}$ is $q = [0,0] \perp [1,1]$. As a map $q: V\rightarrow \F_2$ is defined by
\begin{center}
$q(w,x,y,z)=w^2+wx+x^2+yz$
\end{center}
To show that $Q_{2}\circ D_{4}$ is strongly real using the criterion of Th. \ref{strongly-real-criterion}, 
for each $v \in V$ we have to exhibit some $a \in V$ such that 
$q(a) = q(a - v) = 0$. The following table demonstrates that it is indeed possible.
\begin{center}
\begin{tabular}{|c|c|}
\hline 
$v$ & $a$ \\
\hline 
$(0,0,0,1), (0,0,1,0), (1,1,1,1), (1,0,1,1), (0,1,1,1), (0,0,0,0)$ & $(0,0,0,0)$ \\
$(0,0,1,1)$ & $(0,0,0,1)$ \\
$(1,0,0,0),(0,1,0,0),(1,1,1,0),(1,1,0,1)$ & $(1,1,1,1)$ \\
$(0,1,1,0),(0,1,0,1),(1,1,0,0)$ & $(0,1,1,1)$ \\
$(1,0,1,0),(1,0,0,1)$ & $(1,0,1,1)$ \\
\hline
\end{tabular}
\end{center}
Therefore it follows that groups $Q_2D_4^{(n-1)}$ for $n \geq 2$ are strongly real. 
The only extraspecial $2$-group which is left out it $Q_2$, which is not strongly real. This is because in $Q_2$ there is only one element
of order $2$, which is central.
\hfill $\square$\\ 
\end{proof}

We know the group $Q_{2}$ is real (see, for example \cite{Rose}, p.304). Therefore Prop. \ref{extraspecial-are-strongly-real} 
gives that all extraspecial 2-groups are real.

\section{Examples}
In this section we obtain examples of special $2$-groups which are strongly real but not totally orthogonal, and vice-versa.
We first fix our notation in order to state a criterion for total orthogonality of special $2$-groups. \\

Let $q:V\rightarrow W$ be a quadratic map and $s\in \Hom_{\F_2}(W,\F_2)$. Then $s_*(q):=s\circ q:V\rightarrow \F_2$ is a quadratic form with
 polar form $b_{s_*(q)}:=s\circ b_q:V\times V\rightarrow \F_2$. The form $s_*(q)$ is called the {\it transfer of $q$ by $s$}. 
 If the image of the radical $\rad(b_{s_*(q)})$ under $s_*(q)$ vanishes then $s_*(q)$ induces a regular quadratic form 
 $q_s:V_s:=\frac{V}{\rad(b_{s_*(q)})}\rightarrow \F_2$ defined by
 $q_s(\epsilon_s(x))=s_*(q)(x)$ for every $x\in V$, where $\epsilon_s:V\rightarrow V_s$ is the canonical surjection.

\begin{theorem}[\cite{ObedPaper}, Theorem 3.5]\label{totally-orthogonal-criterion}
 Let $G$ be a special $2$-group with associated quadratic map $q:V\rightarrow W$. Then the following are equivalent.
 \begin{enumerate}
 \item[$i.$] The group $G$ is totally orthogonal.
 \item[$ii.$] For all non-zero $s\in \Hom_{\F_2}(W,\F_2)$ the Arf invariant $\Arf(q_s)$ is trivial.
 \end{enumerate}
 \end{theorem}

If $G$ is an extraspecial $2$-group with associated quadratic form $q$ then $\Hom_{\F_2}(W,\F_2)$ consists only of identity map and
by Th. \ref{totally-orthogonal-criterion} the group $G$ is totally orthogonal if and only if $\Arf(q)$ is trivial. \\

For all $n\in \N$, extraspecial $2$-groups $D_4^{(n)}$ are totally orthogonal because the Arf invariant of the quadratic 
form $q= [0,0] \perp [0,0] \cdots \perp [0,0]$ associated with to the group $D_4^{(n)}$ is trivial. 
Extraspecial $2$-groups $Q_2 D_4^{(n-1)}$ are not totally orthogonal because the Arf invariant of associated the quadratic 
form $q=[1,1] \perp [0,0] \cdots \perp [0,0]$ is not trivial. \\

Therefore to sum up, all extraspecial $2$-groups $Q_2 D_4^{(n-1)}$, $n\geq 2$ are examples of strongly real groups 
which are not totally orthogonal. In fact these groups have exactly one symplectic representation which is of degree
$2^n$ (see \cite{GorensteinBook}). Further, the least order of a strongly real finite group which are not totally orthogonal
is $32$ and $Q_2 \circ D_4$ the only group of order $32$ with this property. We have checked it by computer algebra system
GAP \cite{GAP4}.

Next order in which an example of strongly real group with symplectic representations is found is $64$.

 
\begin{example}
Let $V$ and $W$ be vector spaces over field $\F_2$ and $q(w,x,y,z)=(z^2+wx+wz+xy,wy)$ be a regular quadratic map from $V$ to $W$. We show that the special $2$-group associated to $q$ is strongly real but not totally orthogonal. Let $s: W\rightarrow \F_2$ be the linear map given by $s(w_1,w_2)=w_1+w_2$ for $(w_1,w_2)\in W$. Since 
$s_*(q) : V \to \F_2$ is a regular quadratic form given by $s_*(q)(w,x,y,z)=(z^2+wx+wz+xy+wy)$, the quadratic forms $s_*(q)$ and $q_s$ are same.
The following change of variables in $s_*(q)$ converts it to the form $[1,1]\perp[0,0]$:
\begin{align*}
w &\mapsto w+x+z \\
x &\mapsto x+y \\
y &\mapsto y+w \\
z &\mapsto y+z
\end{align*}
Arf Invariant of $[1,1]\perp[0,0]$ is not trivial. Hence by Th. \ref{totally-orthogonal-criterion} the special $2$-group $G$ associated with quadratic form $q$ is not totally orthogonal. Now Th. \ref{strongly-real-criterion} we show that the special $2$-group $G$ is strongly real. In following table we give $a\in V$ for every $v\in V$ so that criterion of Th. \ref{strongly-real-criterion} is satisfied.
 \begin{center}
\begin{tabular}{|c|c|}
	\hline
$v$ &  $a$      \\
	\hline
$(1,0,0,0),(0,1,0,0),(0,0,1,0),(1,0,0,1),(0,1,1,1)$   & $(0,0,0,0)$ \\
	
$(0,0,0,1),(1,1,0,0),(1,0,1,0),(1,1,1,1)$  & $(1,0,0,0)$ \\
        
$(0,1,1,0),(0,1,0,1)$ & $(0,0,1,0)$\\
	
$(0,0,1,1)$ &  $(0,1,0,0)$\\
	
$(1,1,1,0),(1,0,1,1),(1,1,0,1)$   & $(1,0,0,1)$ \\
	\hline
\end{tabular}  
\end{center}
This is the only special group of order $64$ which is strongly real and not totally orthogonal. Another group of order $64$ which
is strongly real and not totally orthogonal is $\mathcal G = \mu_2 \times( Q_2 \circ D_4)$, where $\mu_2$ is the group of order $2$.
The group $\mathcal G$ is not special.
We have checked using GAP \cite{GAP4} that $G$ and $\mathcal G$ are the only strongly real groups of order $64$ which are not totally
orthogonal. \\
\end{example}

We now give an example of a special $2$-group which is totally orthogonal but not strongly real. 
From Th. \ref{strongly-real-criterion} and Th. \ref{totally-orthogonal-criterion} it is enough to find a quadratic map
$q : V \to W$ between vector spaces over field $\F_2$ such that the Arf invariant $\Arf(q_s)$ is trivial for all non-zero $s\in \Hom_{\F_2}(W,\F_2)$
and for every nonzero $v\in V$ there exists $a\in V$ with $v \neq a$ and $q(a) = q(a - v) = 0$. We assert that such quadratic maps
indeed exist. One such example is the following. \\

\begin{example}
Let $V$ be a $4$-dimensional vector space over $\F_2$ with standard basis 
$$\{(1,0,0,0),(0,1,0,0),(0,0,1,0),(0,0,0,1)\}$$ 
and $W$ be a $3$-dimensional vector space over $\F_2$ with standard basis $$\{(1,0,0),(0,1,0),(0,0,1)\}.$$

\noindent Consider the quadratic form $q: V\rightarrow W$ given by
\begin{align} \label{qform-example}
q(w,x,y,z)=(wx+yz,wy,xy); \quad (w,x,y,z)\in V
\end{align}

\end{example}
\noindent The polar form associated with $q$ is $b_q : V \times V \to W$,
$$
b_q((w_1,x_1,y_1,z_1), (w_2,x_2,y_2,z_2))=(w_1x_2+x_1w_2+y_1z_2+z_1y_2,w_1y_2+y_1w_2,x_1y_2+y_1x_2)
$$
where $(w_1,x_1,y_1,z_1), (w_2,x_2,y_2,z_2)\in V$.

It is straightforward to check that $\rad(b_q)=0$ and $\langle b_q(V\times V)\rangle = W$. Hence by Th. \ref{special-2-group-of-a-quad-map}
there exists a unique special $2$-group $G$ such that $V=Z(G)$ and $\frac{G}{Z(G)} = W$. The order of this group is 
$|V| \times |W| = 128$. 
We shall make explicit computations to show
that $G$ is not strongly real but it is totally orthogonal.

For strong reality take, for example, $v = (1,1,1,1) \in V$. Then we claim that for every $a \in V$ with $q(a) = 0$ we have $q(v - a) \neq 0$.
We first identify all $a \in V$ such that $q(a) = 0$. Let $a = (w,x,y,z) \in V$ be one such vector. Then
$q(w,x,y,z) = 0$ will imply
$$wx + yz = wy = xy = 0$$
If $y \neq 0$ then the above condition forces $x = w = z = 0$. This gives $a = (0, 0, 1, 0)$.
If $y = 0$ then to ensure the above condition one must have either $w = 0$ or $x = 0$. Therefore we conclude that
$a \in \{(0,0,0,0),(0,1,0,0), (1,0,0,0),(0,0,0,1),(0,1,0,1), (1,0,0,1)\}$.
For each of these possibilities for $a$ we compute $q(a-v)$.

\begin{center}
\begin{tabular}{|c|c|c|}
\hline 
$a$ & $a-v$ & $q(a-v)$ \\
\hline 
$(0,0,1,0)$ & $(1,1,0,1)$ & $(1,0,0)$ \\
$(0,0,0,0)$ & $(1,1,1,1)$ & $(0,1,1)$ \\
$(0,1,0,0)$ & $(1,0,1,1)$ & $(1,1,0)$ \\
$(1,0,0,0)$ & $(0,1,1,1)$ & $(1,0,1)$ \\
$(0,0,0,1)$ & $(1,1,1,0)$ & $(1,1,1)$ \\
$(0,1,0,1)$ & $(1,0,1,0)$ & $(0,1,0)$ \\
$(1,0,0,1)$ & $(0,1,1,0)$ & $(0,0,1)$ \\
\hline 
\end{tabular}\\
\end{center} 

This table confirms that for $v = (1,1,1,1)$ there is no $a \in V$ such that $q(a) = q(a-v) = 0$ and from Th. \ref{strongly-real-criterion}
we conclude that $G$ is not strongly real. \\

Now we show that the special $2$-group $G$ associated to the quadratic form $q$ as in (\ref{qform-example}) is totally orthogonal. Since $\dim_{\F_2}(W,\F_2)=3$,
there exist exactly $7$ non-zero $\F_2$-linear maps from $W$ to $\F_2$, which are following
$$s_n(x,y,z)=ix+jy+kz; \quad \quad(x,y,z)\in W, 1\leq n \leq 7,$$
where $n =  4i+2j+k$ is the binary expansion of $n \in \{1, 2, \cdots, 7\}$. We write various transfer maps of $q$

\begin{align*}
s_{1_*}(q)(w,x,y,z)&=xy\\
s_{2_*}(q)(w,x,y,z)&=wy\\
s_{3_*}(q)(w,x,y,z)&=wy+xy\\
s_{4_*}(q)(w,x,y,z)&=wx+yz\\
s_{5_*}(q)(w,x,y,z)&=wx+yz+xy\\
s_{6_*}(q)(w,x,y,z)&=wx+yz+wy\\
s_{7_*}(q)(w,x,y,z)&=wx+yz+wy+xy,
\end{align*}
where $(w,x,y,z)\in V$. By suitable linear changes of variables each of above quadratic forms is isometric to either $q_1:V\rightarrow \F_2$ defined by $q_1(w,x,y,z)=wy$
or $q_2:V\rightarrow \F_2$ defined by $q_1(w,x,y,z)=wx+yz$. \\

Now $\rad(b_{q_1}) = \langle (0,1,0,0),(0,0,0,1) \rangle$ and $\frac{V}{\rad(b_{q_1})} = \langle (1,0,0,0),(0,0,1,0) \rangle$ so $q_1$
induces regular quadratic form $q^{\prime}_1:\frac{V}{\rad(b_{q_1})}\rightarrow \F_2$ defined by $q^{\prime}_1(\alpha,\beta)=\alpha\beta$,
where $(\alpha,\beta)\in \frac{V}{\rad(b_{q_1})}$. Now $\Arf(q_1) = \Arf(q^{\prime}_1)=0$. 
On the similar lines, since $\rad(b_{q_2}) = 0$ the quadratic form $q_2$ is regular and $\Arf(q_2)=0$. \\

As a consequence, for all $s \in \Hom_{\F_2}(W,\F_2)$ the Arf invariant of the transfer $q_s$ is trivial and by Th. \ref{totally-orthogonal-criterion}
the group $G$ is totally orthogonal. \\

\remark We remark that the smallest order in which a totally orthogonal special $2$-group which is not strongly real exists is $128$. 
We have checked using GAP \cite{GAP4} that the smallest totally orthogonal group which is not strongly real is of order $64$. That group, though, is not
a special $2$-group.

\bibliographystyle{amsalpha}
\bibliography{references}

\providecommand{\bysame}{\leavevmode\hbox to3em{\hrulefill}\thinspace}
\providecommand{\MR}{\relax\ifhmode\unskip\space\fi MR }
\providecommand{\MRhref}[2]{%
  \href{http://www.ams.org/mathscinet-getitem?mr=#1}{#2}
}
\providecommand{\href}[2]{#2}
\begin{thebibliography}{GAP08}

\bibitem[Arf41]{Arf}
Cahit Arf, \emph{Untersuchungen {\"u}ber quadratische {F}ormen in {K}{\"o}rpern
  der {C}harakteristik 2. {I}}, J. Reine Angew. Math. \textbf{183} (1941),
  148--167. \MR{0008069 (4,237f)}

\bibitem[GAP08]{GAP4}
The GAP~Group, \emph{{GAP -- Groups, Algorithms, and Programming, Version
  4.4.12}}, 2008.

\bibitem[Gor82]{GorensteinBook}
Daniel Gorenstein, \emph{Finite simple groups}, University Series in
  Mathematics, Plenum Publishing Corp., New York, 1982, An introduction to
  their classification. \MR{698782 (84j:20002)}

\bibitem[Gow85]{GowOrthogonal}
R.~Gow, \emph{Real representations of the finite orthogonal and symplectic
  groups of odd characteristic}, J. Algebra \textbf{96} (1985), no.~1,
  249--274. \MR{808851 (87b:20015)}

\bibitem[HL04]{HL}
Detlev~W. Hoffmann and Ahmed Laghribi, \emph{Quadratic forms and {P}fister
  neighbors in characteristic 2}, Trans. Amer. Math. Soc. \textbf{356} (2004),
  no.~10, 4019--4053 (electronic). \MR{2058517 (2005e:11041)}

\bibitem[KS11]{AmitAnupam}
Amit Kulshrestha and Anupam Singh, \emph{Real elements and schur indices of a
  group}, Math. Student \textbf{80} (2011), no.~1-4, 73--84.

\bibitem[Pfi95]{Pfister}
Albrecht Pfister, \emph{Quadratic forms with applications to algebraic geometry
  and topology}, London Mathematical Society Lecture Note Series, vol. 217,
  Cambridge University Press, Cambridge, 1995. \MR{1366652 (97c:11046)}

\bibitem[Ros94]{Rose}
John~S. Rose, \emph{A course on group theory}, Dover Publications Inc., New
  York, 1994, Reprint of the 1978 original [Dover, New York; MR0498810 (58
  \#16847)]. \MR{1298629}

\bibitem[Sch85]{Scharlau}
W.~Scharlau, \emph{Quadratic and hermitian forms}, Springer-Verlag, 1985.

\bibitem[Ser77]{SerreRepnBook}
Jean-Pierre Serre, \emph{Linear representations of finite groups},
  Springer-Verlag, New York, 1977, Translated from the second French edition by
  Leonard L. Scott, Graduate Texts in Mathematics, Vol. 42. \MR{0450380 (56
  \#8675)}

\bibitem[Wei94]{Weibelbook}
Charles~A. Weibel, \emph{An introduction to homological algebra}, Cambridge
  Studies in Advanced Mathematics, vol.~38, Cambridge University Press,
  Cambridge, 1994. \MR{1269324 (95f:18001)}

\bibitem[Wil09]{Wilson}
Robert~A. Wilson, \emph{The finite simple groups}, Graduate Texts in
  Mathematics, vol. 251, Springer-Verlag London Ltd., London, 2009. \MR{2562037
  (2011e:20018)}

\bibitem[Won66]{WonenburgerOrthogonal}
Mar{\'i}a~J. Wonenburger, \emph{Transformations which are products of two
  involutions}, J. Math. Mech. \textbf{16} (1966), 327--338. \MR{0206025 (34
  \#5850)}

\bibitem[Zah08]{Obedthesis}
Obed~Ntabuhashe Zahinda, \emph{Ortho-ambivalence des groupes finis}, Ph.D.
  thesis, Universit{\'e} catholique de Louvain, 2008.

\bibitem[Zah11]{ObedPaper}
\bysame, \emph{Les 2-groupes sp{\'e}ciaux totalement orthogonaux}, Comm.
  Algebra \textbf{39} (2011), no.~2, 435--451. \MR{2773312 (2012f:20020)}

\end{thebibliography}
\end{document}